\documentclass[journal]{IEEEtran}
\IEEEoverridecommandlockouts                              
\usepackage{url}

\usepackage{amsmath, amssymb, bbm, xspace}
\usepackage{epsfig}
\usepackage{longtable}
\usepackage{color}
\usepackage{mathrsfs}
\usepackage{comment}

\usepackage{courier}

\newtheorem{theorem}{Theorem}[section]

\newtheorem{lemma}[theorem]{Lemma}

\newtheorem{proposition}[theorem]{Proposition}

\newtheorem{definition}{Definition}[section]

\def\bkE{{\rm I\kern-.17em E}}
\def\bk1{{\rm 1\kern-.17em l}}
\def\bkD{{\rm I\kern-.17em D}}
\def\bkR{{\rm I\kern-.17em R}}
\def\bkP{{\rm I\kern-.17em P}}
\def\Cov{{\rm Cov}}

\def\bkZ{{\bf{Z}}}

\def\bfzero{{\bf 0}}

\def\bkE{{\rm I\kern-.17em E}}
\def\bk1{{\rm 1\kern-.17em l}}
\def\bkD{{\rm I\kern-.17em D}}
\def\bkR{{\rm I\kern-.17em R}}
\def\bkP{{\rm I\kern-.17em P}}

\makeatletter
\newcommand{\pushright}[1]{\ifmeasuring@#1\else\omit\hfill$\displaystyle#1$\fi\ignorespaces}
\newcommand{\pushleft}[1]{\ifmeasuring@#1\else\omit$\displaystyle#1$\hfill\fi\ignorespaces}
\makeatother

\def\bkZ{{\bf{Z}}}
\def\b12{(\beta_1,\beta_2)}

\newenvironment{example}{{\noindent \bf Example}}{\hfill $\square$\hspace{-4.5pt}\vspace{6pt}}
\newcounter{example}
\renewcommand{\theexample}{\thesection.\arabic{example}}
\newenvironment{examplec}[1][]{\refstepcounter{example}
\par\medskip \noindent%
   \textbf{Example~\theexample. #1} \rmfamily}{\hfill $\square$   \hspace{-4.5pt} \vspace{6pt}}

\newcounter{remark}
\renewcommand{\theremark}{\thesection.\arabic{remark}}
\newenvironment{remarkc}[1][]{\refstepcounter{remark}
\noindent{\itshape Remark~\theremark. #1} \rmfamily}{\hspace*{\fill}~$\square$\vspace{0pt}}

\def\t{^\top}
\def\Bscr{\mathscr{B}}

\def\Ebb{\mathbb{E}}
\newlength{\noteWidth}
\long\def\notes#1{\ifinner
{\tiny #1}
\else
\marginpar{\parbox[t]{\noteWidth}{\raggedright\tiny #1}}
\fi\typeout{#1}}

 \def\notes#1{\typeout{read notes: #1}} 

\newcommand{\I}[1]{\mathbb{I}_{\{#1\}}}

\newcommand{\ie}{i.e.\@\xspace} 
\newcommand{\eg}{e.g.\@\xspace} 

\newcommand{\Real}{\ensuremath{\mathbb{R}}}
\newcommand{\inv}{^{-1}}

\newcommand{\minimize}[1]{\displaystyle\minim_{#1}}
\newcommand{\minim}{\mathop{\hbox{\rm min}}}

\def\OPT{{\rm OPT}}

\def\Ebb{\mathbb{E}}

\def\Nbb{{\mathbb{N}}}

\def\limn{\ds \lim_{n \rightarrow \infty}}

\def\det{\mathop{{\rm det}}}
\def\diag{\mathop{\hbox{\rm diag}}}
  
\def\exp{\mathop{\hbox{\rm exp}}}

\def\range{{\rm range}}

\def\half  {{\textstyle{1\over 2}}}

\def\limk{\lim_{k\to\infty}}

\def\norm#1{\|#1\|}

\def\spose#1{\hbox to 0pt{#1\hss}}
\def\sub#1{^{\null}_{#1}}
\def\text #1{\hbox{\quad#1\quad}}

\def\Escr{\mathcal{E}}

\def\nthinsp{\mskip -2   mu}

\def\superstar{^{\raise 0.5pt\hbox{$\nthinsp *$}}}
\def\SUPERSTAR{^{\raise 0.5pt\hbox{$*$}}}

\def\lamstarT {\lambda^{\raise 0.5pt\hbox{$\nthinsp *$}T}}

\def\Bscr{{\cal B}}
\def\Fscr{{\cal F}}

\def\Lscr{{\cal L}}

\def\Vscr{{\cal V}}

\def\Nscr{{\cal N}}
\def\Rscr{{\cal R}}
\def\Gscr{{\cal G}}

\def\th{^{\rm th}}

\def\xbar{\skew{2.8}\bar x}

\def\aur{\;\textrm{and}\;}

\def\eef{\;\textrm{if}\;}
\def\ow{\;\textrm{otherwise}\;}

\def\non{\nonumber}

\let\forallnew\forall
\renewcommand{\forall}{\forallnew\ }
\let\forall\forallnew

\def\ds{\displaystyle}
		\def\bkE{{\rm I\kern-.17em E}}
		\def\bk1{{\rm 1\kern-.17em l}}
		\def\bkD{{\rm I\kern-.17em D}}
		\def\bkR{{\rm I\kern-.17em R}}
		\def\bkP{{\rm I\kern-.17em P}}
		\def\bkY{{\bf \kern-.17em Y}}
		\def\bkZ{{\bf \kern-.17em Z}}
		\def\bkC{{\bf  \kern-.17em C}}

{\begin{list}{}%
         {\setlength{\leftmargin}{#1}}%
         \item[]%
}
{\end{list}}

		\def\bsp{\begin{split}}
		\def\beq{\begin{eqnarray}}
		\def\bal{\begin{align*}}
		\def\bc{\begin{center}}
		\def\be{\begin{enumerate}}
		\def\bi{\begin{itemize}}
		\def\bs{\begin{small}}
		\def\bS{\begin{slide}}
		\def\ec{\end{center}}
		\def\ee{\end{enumerate}}
		\def\ei{\end{itemize}}
		\def\es{\end{small}}
		\def\eS{\end{slide}}
		\def\eeq{\end{eqnarray}}
		\def\eal{\end{align*}}
		\def\esp{\end{split}}
		\def\qed{ \vrule height7.5pt width7.5pt depth0pt}  

		\def\problemsmall#1#2#3#4{\fbox
		 {\begin{tabular*}{0.47\textwidth}
			{@{}l@{\extracolsep{\fill}}l@{\extracolsep{6pt}}l@{\extracolsep{\fill}}c@{}}
				#1 & $\minimize{#2}$ & $#3$ & $ $ \\[5pt]
					  $\sub\ $ &    & $#4$ & $ $
			\end{tabular*}}
			}

	\def\cp2problem#1#2#3#4{\fbox
		 {\begin{tabular*}{0.9\textwidth}
			{@{}l@{\extracolsep{\fill}}l@{\extracolsep{6pt}}l@{\extracolsep{\fill}}c@{}}
				#1 & & $#4 $ 
			\end{tabular*}}}

\newcommand{\pmat}[1]{\begin{pmatrix} #1 \end{pmatrix}}

		\def\bkE{{\rm I\kern-.17em E}}
		\def\bk1{{\rm 1\kern-.17em l}}
		\def\bkD{{\rm I\kern-.17em D}}
		\def\bkR{{\rm I\kern-.17em R}}
		\def\bkP{{\rm I\kern-.17em P}}
		
		\def\bkZ{{\bf{Z}}}

\newcommand {\beeq}[1]{\begin{equation}\label{#1}}
\newcommand {\eeeq}{\end{equation}}
\newcommand {\bea}{\begin{eqnarray}}
\newcommand {\eea}{\end{eqnarray}}

\def\texitem#1{\par\smallskip\noindent\hangindent 25pt
               \hbox to 25pt {\hss #1 ~}\ignorespaces}

\def\bsp{\begin{split}}
		\def\beq{\begin{eqnarray}}
		\def\bal{\begin{align*}}
		\def\bc{\begin{center}}
		\def\be{\begin{enumerate}}
		\def\bi{\begin{itemize}}
		\def\bs{\begin{small}}
		\def\bS{\begin{slide}}
		\def\ec{\end{center}}
		\def\ee{\end{enumerate}}
		\def\ei{\end{itemize}}
		\def\es{\end{small}}
		\def\eS{\end{slide}}
		\def\eeq{\end{eqnarray}}
		\def\eal{\end{align*}}
		\def\esp{\end{split}}
		\def\qed{ \vrule height7.5pt width7.5pt depth0pt}  

\usepackage{amsmath, amssymb, xspace}
\usepackage{epsfig}
\usepackage{longtable}
\usepackage{color}
\usepackage{mathrsfs}
\usepackage{subfig}
\newenvironment{proof}[1][]{{\noindent \emph {Proof} #1: }}{\hfill \qed \vspace{3pt}\\ }

\def\Nscr{{\cal N}}

\def\sub{\hbox{\rm s.t}}

\ifCLASSINFOpdf
\else
\fi
\def\Hbf{{\bf H}}

\begin{document}
%

\author{Ankur A. Kulkarni\thanks{Ankur is with the Systems and Control Engineering group  
at the Indian Institute of Technology Bombay, Mumbai, India 400076. He can be contacted at \texttt{kulkarni.ankur@iitb.ac.in}. This work was presented in part at the IEEE Conference on Decision and Control, 2015, held in Osaka, Japan~\cite{kulkarni2015approximately}.}}



\title{Near-Optimality of Linear Strategies for Static Teams with `Big' Non-Gaussian Noise}

\maketitle

\begin{abstract}
We study stochastic team problems with static information structure where we assume controllers have linear information and quadratic cost but allow the noise to be from a non-Gaussian class. When the noise is Gaussian, it is well known that these problems admit linear optimal controllers. We show that for such linear-quadratic static teams with any log-concave noise, if the length of the noise or data vector becomes large compared to the size of the team and their observations, then linear strategies approach optimality for `most' problems.
The quality of the approximation improves as length of the noise vector grows and the class of problems for which the approximation is asymptotically not exact approaches a set of measure zero.
We show that if the optimal strategies for problems with log-concave noise converge pointwise, they do so to the (linear) optimal strategy for the problem with Gaussian noise. And we derive an asymptotically tight error bound on the difference between the optimal cost for the non-Gaussian problem and the best cost obtained under linear strategies. 
\end{abstract}


%
\IEEEpeerreviewmaketitle

\section{Introduction}
The assumption of Gaussian noise is ubiquitous in engineering. A common justification for this assumption is the central limit theorem -- one assumes that the noise present in a system is the aggregate, of a large number of small, independent and identically distributed disturbances, which by the central limit theorem converges in distribution to a normal distribution. 

In stochastic control the assumption of Gaussian noise together with further assumptions of quadratic cost and linear dynamics (the so-called classical `LQG' setting) is the source of a number of clean and elegant results. If the problem at hand is filtering, then the optimal estimator admits a beautiful recursive form given by the Kalman filter. Moreover, for a control problem the optimal controller is linear in the information and a separation principle holds whereby the optimal controller can obtained as the superposition of the optimal controller of the linear quadratic regulator, and the Kalman filter, both of which can be designed independently. 

A landmark observation in this field, due to  Witsenhausen~\cite{witsenhausen_counterexample_1968}, was that these conclusions, in particular, the linearity of the optimal controller, cease to be automatic for problems with \textit{non-classical} information structures. Ever since, the quest for linearity has acquired a life of its own~\cite{ho_teams_1978,bansal87stochastic}, and questions such as tractability of these problems, their convexity and the linearity of the optimal controllers have been probed from various vantage points; see, \eg,~\cite{papadimitriou85intractable,mitter1999information,rotkowitz2006characterization,mahajan2012information,
kulkarni2014optimizer}.

 This paper is motivated by the draw  of two recent developments. First, an ongoing technological revolution  has enabled the collection and storage of vast amounts of data~\cite{bigdataharvard}. A conceivable consequence of this is a situation where partial information of this high-dimensional data is available to a small number of controllers that have to achieve a common objective. Second,  fascinating new mathematics~\cite{klartag2010high} has been discovered revealing far-reaching connections between convexity and probability, and  leading to general and powerful versions of the central limit theorem.

Motivated by this, the present paper studies stochastic control problems \textit{without} the assumption of Gaussian noise but in the regime that the noise vector is high dimensional. Our focus is on team problems with \textit{static} information structure, a problem which has been extensively studied in the LQG setup~\cite{radner1962team,ho1980team} wherein it is known to admit linear optimal controllers. We retain the other assumptions from the LQG realm -- we assume that the cost is quadratic and that observations are linear in the environmental randomness -- but allow this randomness to be from a non-Gaussian class. We find that the quest for linearity has a fruitful end even in this problem class. Our main results show that for `most' problems such as these, linear controllers are near-optimal. The nearness improves as the length of the vector of environmental randomness grows. 

The class of noise distributions we consider are those with log-concave densities. A density $f:\Real^n\rightarrow [0,\infty)$ is log-concave if $f(x)\equiv \exp(-G(x))$ for some convex function $G.$ 
This class includes Gaussians and many other distributions, \eg, the exponential, Gamma, beta, chi-square distributions are all log-concave. We refer the reader to the survey~\cite{wellner2014log} for more. They enjoy many useful properties, \eg, log-concavity is preserved under marginals 
and more generally under linear transformations. The surprising central limit theorem is that most low-dimensional linear transformations (more, precisely low-dimensional projections) of a random vector with a log-concave density are in fact (nearly) Gaussian (Theorem~\ref{thm:klartag} of Eldan and Klartag below). 

Using the central limit theorem to justify the assumption of Gaussian noise amounts to claiming that the optimal cost is weakly continuous with respect to the noise. However, in a special case of our problem, the mean square estimation problem, it is well known that the minimum mean square error (\ie, the optimal cost)  is neither weakly upper semicontinuous, nor weakly lower semicontinuous~\cite{wu2012functional}. Thus, more specifics about the kind of the central limit theorem being used and other facts about the LQG setting would have to be exploited for showing weak continuity of the optimal cost.


\subsection{Contributions}
Our main contribution here is in showing that for ``most'' problems one does have weak continuity. 
We show an asymptotically tight error bound on the difference between the optimal cost for the non-Gaussian problem and the best cost obtained under linear strategies. Remarkably, these results are true for \textit{any} sequence of noise vectors having log-concave density  (in particular, the noise need not be drawn from any particular parametrized class such as exponentials) and for any strictly convex quadratic cost. 
As such these results apply to canonical problems such as mean square estimation and variance reduction. Indeed our results agree with the existing evidence\footnote{I thank the Senior Editor Prof James Spall for pointing this out.} that the Kalman filter produces useful results even in the case of log-concave non-Gaussian noise~\cite{spall1995kantorovich}. 
Since problems with classical and partially nested information structures can be reduced to those with static information structures~\cite{ho80another}, our results could also be applied to problems with these dynamic information structures. However, the precise nature of this extension is not worked out in this paper. We also show that as the dimension of the noise grows, 
if the solution of the non-Gaussian problem converges pointwise, it does so to the (linear) optimal solution of the Gaussian problem.

The theme of partial observations of a long noise (or state) vector has been around for a while, a prime example being the applications to the power grid. To the best of the author's knowledge, statistical implications of this theme  have not been exploited in the context of stochastic control. 
This paper is an improvement over its conference version~\cite{kulkarni2015approximately}. Specifically, the bounds in the present paper are asymptotically tight whereas the asymptotics of the earlier bounds were not unconditionally tight. This tightening is accomplished in the present paper by the derivation of new bounds and the introduction of new proof approaches (Section~\ref{sec:error} elaborates this). Moreover, several claims which were earlier shown to hold with high probability have now been refined and are shown to hold almost surely. 

\subsection{Organization}
The paper is organized as follows.  The initial part of the paper comprises of formalization of the problem definition and the derivation of supporting results required for the final theorems. Specifically, we need to define a sequence or \textit{ensemble} of team problems where the dimensions of certain parameters grow with the length of the noise vector. In Section~\ref{sec:probdef}, we introduce these preliminaries and define the problem. 
In Section~\ref{sec:convergence}, we introduce the central limit theorem that this paper rests on (Theorem~\ref{thm:klartag}) and derive the supporting results needed for our final claims. Section~\ref{sec:main} contains our main results on team problems. 
Section~\ref{sec:uniform} contains some stronger convergence results, which, though not required for our main results, we think may be useful to other researchers. 
We conclude in Section~\ref{sec:conc}.


\section{Problem definition} \label{sec:probdef}

\subsection{Preliminaries}
We assume that there is a primeval probability space $(\Omega,\Fscr,P)$  on which all random variables are defined. When a random vector (or matrix) $X$ takes values in a space $S$, we write it as ``$X \in S$". When a statement holds almost surely with respect to the a law of a random vector $X$, we write that it holds ``$X$-a.s.'' 
For any random vector $X$, let $\Rscr(X)$ denote its range, \ie, the values it is allowed to take. 
Random variables in this paper will often take values in a subspace of a Euclidean space whereby $\Rscr(X)$ is in general not the entire ambient space. 
For a random vector $X$ 
 (having a density with respect to the Lebesgue measure on $\Rscr(X)$), we denote its density by $f_X: \Rscr(X) \rightarrow [0,\infty).$ A random  vector $X \in \Real^n$ is said to be \textit{isotropic} if $\Ebb[X]=0$ and its covariance matrix $\Cov(X)=I_n$, where $I_n$ is the $n\times n$ identity matrix. 
We denote the $n$ dimensional standard normal density by $\varphi_n: \Real^n \rightarrow [0,\infty).$ Recall, $ \varphi_n(x) =\frac{1}{(2\pi)^{n/2}}\exp(-\norm{x}^2/2)$ for all $x\in \Real^n.$ The \textit{range} of a matrix $M\in \Real^{m\times n}$ is its column space and denoted  $\range(M)=\{y \in \Real^m | y=Mx, x \in \Real^n\}.$

The following information is sourced mainly from~\cite{chikuse2003statistics}. Let $G_{n,\ell}$ denote the set (called the Grassmann manifold) of all $\ell$ dimensional subspaces of $\Real^n$. 
A matrix $R \in \Real^{n \times \ell}$ is called \textit{orthonormal} if $R\t R = I_\ell$. Let $V_{n,\ell}$ denote the set (called Stiefel manifold~\cite{chikuse2003statistics,muirhead2009aspects})  of $n \times \ell$ orthonormal matrices.
The orthogonal projection of $x \in \Real^n$ on the subspace $E=\range(R)=\{z\ |\ z= Ry, y \in \Real^\ell\} \in G_{n,\ell}$ for some $R \in V_{n,\ell}$, is denoted by $\Pi_E(x).$ It is the vector $Ry,$ where $y$ is given by, 
\[\arg \min_{y' \in \Real^\ell} \norm{x - Ry'}.\]
It is easy to see that  $ \Pi_E(x)=RR\t x$ and that $R\t x$ is the vector of \textit{coordinates} of the orthogonal projection with respect to the basis given by the columns of $R.$ Let $P_{n,\ell}$ be the set of all $n \times n$ idempotent matrices (\ie, matrices $P \in \Real^{n\times n}$ such that $P^2=P$) of rank $\ell$.
To each subspace $\Vscr \in G_{n,\ell}$ corresponds a unique projection matrix $P \in P_{n,\ell}$ whose range is $\Vscr$. If $R \in V_{n,\ell}$ is such that the columns of $R$ span $\Vscr,$ then $P=RR\t.$ 

$G_{n,\ell}$ can be endowed with a unique rotationally invariant probability measure; we denote this measure by $\mu_{n,\ell}.$ This measure defines a `uniform' distribution on $G_{n,\ell}.$ To sample from $\mu_{n,\ell}$ one may equivalently sample from the uniform distribution on $P_{n,\ell}.$
There is an intimate connection between the uniform distribution on $V_{n,\ell}$ and that on $P_{n,\ell}.$ Specifically,
\begin{lemma}[Theorem 2.2.2~\cite{chikuse2003statistics}] \label{lem:pnl} 
Let $R \in \Real^{n \times \ell}$. Then $P=RR\t$ is distributed uniformly on $P_{n,\ell}$ if and only if $R$ is distributed uniformly on $V_{n,\ell}.$
\end{lemma}
%

Now suppose $X$ is a random vector in $\Real^n$ and $R \in \Real^{n \times \ell}$ is an orthonormal matrix. The joint density of the projection $RR\t X$ and the joint density of $R\t X$ (the ``coordinates'' of the projection mentioned above) are in essence the same. We show this in the lemma below.

\begin{lemma} \label{lem:rrt} 
 Let $n,\ell \in \Nbb, n\geq \ell,$ and let $R \in V_{n,\ell}$. Then, $f_{R\t X}(R\t x) = f_{RR\t x}(RR\t x) $ for all $x \in \Real^n.$
\end{lemma}
\begin{proof} Let $Y = RR\t X$. Let $P=[R\ \bar{R}]$ where $\bar{R} \in \Real^{n \times n-\ell}$ is an orthonormal matrix, orthogonal to $R$. 

Let  $V= P\t Y = (R\t X, 0)$. By a change of variables,
\[f_{V}(v) = \frac{f_{Y}(P v)}{|\det(P\t)|}=f_{Y}(Pv), \ \  \forall v =(R\t x,0), x\in \Real^n.\]
Therefore $f_V(R\t x,0) = f_Y(RR\t x)$ for all $x \in \Real^n.$ But $V=(R\t X,0)$, whereby $f_V(R\t x,0) = f_{R\t X}(R\t x)$ for all $x \in \Real^n.$ Thus, we get the result.
\end{proof}

Note that since $R\in V_{n,\ell}$  is full rank, the range of $R\t$ is $\Real^\ell$.

\subsection{An ensemble of static team problems} \label{sec:ensemble} 
In this section we define the stochastic control problems we are interested in. 
We consider the following \textit{static} team problem. 
\begin{align*}
	\problemsmall{NG}
	{\gamma}
	{\Ebb\left[ L(u,\xi)\right] }
				 {\begin{array}{r@{\ }c@{\ }l}
			u^i&=\gamma^i(y_i),\\ 
			y_i &= H_i \xi,
	\end{array}} 
\end{align*}
where $H_i \in \Real^{\ell_i \times n}$ is a deterministic matrix of full row rank, $\xi \in \Real^n$ is a random vector, $u\triangleq (u^1,\hdots,u^m) \in \Real^m,$ is the vector of actions of $m$ players, each taking scalar actions (this is without loss of generality), 
and the decision variable is $ \gamma \triangleq (\gamma^1,\hdots,\gamma^m)$,  where each $\gamma^i: \Real^{\ell_i} \rightarrow \Real$, $i=1,\hdots,m$ is a measurable function. 
The function $L: \Real^{m+n}\rightarrow [0,\infty)$ takes the form,
\begin{equation}
L(u,\xi) = \half u\t Qu + u\t S\xi +q(\xi), \label{eq:l} 
\end{equation}
where $Q \succ 0$ and symmetric, $S\in \Real^{m\times n}$ is deterministic and $q(\xi) =\half \xi\t S\t Q\inv S\xi$. This value of $q(\xi)$ ensures that $L(u,\xi)\geq 0$ for all $(u,\xi) \in \Real^{m+n}.$ The above problem has a static information structure~\cite{ho_teams_1978} since the information of each player is not affected by  the actions of any player. 

Following the `input-output' modeling framework of Ho~\cite{ho1980team}, the vector $\xi$ comprises of all the environmental randomness in the problem, including initial state and noise. 
Importantly, we \textit{do not} assume that $\xi$ is Gaussian (indicated by the name NG standing for ``non-Gaussian''). We will assume only that $\xi$ is isotropic and that it has a log-concave density. Of course, the standard Gaussian vector also satisfies these assumptions. To motivate our setting we consider the following example.
\begin{examplec}
Consider the problem of coordinating distributed renewable generation. Suppose $m$ renewable wind generators are located at distinct locations in a country. Let $\xi$ be a random vector that denotes global meteorological conditions (pressure, temperature, humidity and so on) at all locations on the earth. We think of $\xi$ as comprising of microstates that are not directly observable. The generation produced by generator $i$ depends on relevant local weather conditions which is a function macrostates (for instance wind speed and direction) at location $i$, denoted  $y_i$. These macrostates $y_i$ are a function of the microstates $\xi$. We assume that the dependence of $y_i$ on $\xi$ takes a linear form, \ie, $y_i=H_i\xi.$ Without loss of generality one may take $H_i$ to be full row rank, since observations from rows of $H_i$ that are linearly dependent on other rows provide no additional information. Note that in general $H_i$ is non-sparse. Generators have to choose their generation levels $u^i$ such that they are adapted to $y_i$, \ie, $u^i=\gamma^i(y_i)$ for each $i$, in order to collectively minimize a quadratic cost. For instance generators may want to collectively minimize deviation from demand. Assuming that the total demand is also dependent on meteorological conditions, say it is given as $d\t \xi$, and considering the $L^2$ deviation, the problem becomes,
\[\min_{\gamma^1,\hdots,\gamma^m}\Ebb \bigg [ \norm{\sum_i u^i -d\t \xi}^2 \bigg]. \]  
We note two important points. First, components of $\xi$ are correlated, since meteorological conditions of nearby locations are interrelated. Second, $\xi$ can be extremely high dimensional, and in general much larger than the number of generation locations and dimension of individual $y_i$'s.
\end{examplec}

As in the above example, we will study problems NG in the setting where $n \gg \ell$ where $\ell\triangleq \bar{\ell}+m$ and $\bar{\ell}=\sum_i\ell_i$. As $n$ varies, so must the problem parameters $S $ and $H_i, i=1,\hdots,m$. Consequently, we need to define an ensemble of team problems that specifies the evolution of these parameters with $n.$
To this end, let $Z=Z_n \in \Real^{\ell \times n}$ denote the matrix given by
\begin{equation}
Z \triangleq \left[\begin{array}{c}
S \\ H_1\\ \vdots \\ H_m
\end{array} \right]. \label{eq:Z} 
\end{equation}
Since $n >\ell$, $Z$ can have rank at most $\ell.$ Write $Z$ as 
$$Z= WR\t$$ where $R \in V_{n,\ell}$ and $W \in \Real^{\ell \times \ell}.$ If $Z$ indeed has rank $\ell,$ then $W$ is nonsingular and the row space of $Z$ is equal to the row space of $R\t$. 
When considering the ensemble of team problems, we will fix matrices $Q,W$ (and thereby the dimension $\ell$) and let dimension $n$ and the matrix $R$ vary.

\begin{definition} Let $m \in \Nbb, \ell_i \in \Nbb, i=1,\hdots,m$ $Q \in \Real^{m\times m}, Q\succ 0$ and $W\in \Real^{\ell \times \ell}$ be fixed numbers or matrices. Let  $\{R_n\}$ be a sequence of random orthogonal matrices that are independent of each other and of all other random variables such that each $R_n $ is distributed uniformly on $V_{n,\ell}.$ Let $\{\xi_n\}$ be a sequence of random  vectors where each $\xi_n$ is $\Real^n$-valued and has a density with respect to the Lebesgue measure on $\Real^n.$ 
An \textit{ensemble} of static team problems with parameters $Q,W,\{R_n\},\{\xi_n\}$
is a sequence of problems NG, one for each $n\geq \ell$ where the problem parameters are given by
\[\Real^{\ell \times n} \ni Z=Z_n = WR_n\t, \  q(v)= \half v\t S\t Q\inv Sv \  \forall v \in \Real^n,\]
and the environmental randomness $\xi = \xi_n$, 
for each $n \geq \ell.$ The ensemble where the environmental randomness is a Gaussian vector $\zeta=\zeta_n \in \Real^n$, $\zeta_n \sim \Nscr(0,I_n)$ is referred to as the Gaussian ensemble.
\end{definition} 

\begin{remarkc} Another interpretation of the sequence of team problems defined above is as follows. We have a team problem with fixed matrices $S, H_1,\hdots,H_m$. For each $n\in \Nbb,$ the environmental noise is an orthonormal projection $R_n\t \xi_n$ of a random vector $\xi_n \in \Real^n,$ where $R_n$ is distributed uniformly on $V_{n,\ell}.$ Thus a (random) sequence of team problems is generated as we vary $n.$ All our results can also applied with this setting.
\end{remarkc}

The statements we make in this paper will be for each ensemble. In particular, they will pertain to fixed sized teams and fixed sized observations with fixed matrices $Q,W.$ 

\section{Convergence of densities} \label{sec:convergence} 

Our results make extensive use of convergence results of densities of random vectors. This section establishes these results; proofs are relegated to the Appendix. 
The core result that our conclusions rely on is a recent pointwise estimate of the density of the \textit{projection} of an isotropic random vector with a log-concave density. Recall that a density $f:\Real^n \rightarrow [0,\infty)$ is log-concave if $f(x) \equiv \exp(-G(x))$ for some convex function $G:\Real^n \rightarrow \Real.$

\begin{theorem}[Eldan and Klartag~\cite{eldan2008pointwise}] \label{thm:klartag} 
Let $X$ be an isotropic random vector in $\Real^n$ with a log-concave density and let $1 \leq \ell \leq n^{c_1}$ be an integer. Then there exists a  subset $\Escr \subseteq G_{n,\ell}$ with $\mu_{n,\ell}(\Escr) \geq 1-C\exp(-n^{c_2}) $ such that for all $E \in \Escr,$ the following holds. Let $f_E$ denote the density of $\Pi_E(X).$ Then for all $x \in E$ with $\norm{x} \leq n^{c_4}$,
\begin{equation}
\left \lvert\frac{f_E(x)}{\varphi_\ell(x)}-1\right\rvert \leq \frac{C}{n^{c_3}}, \label{eq:tvmain} 
\end{equation}
where $\varphi_\ell(x)=\frac{1}{(2\pi)^{\ell/2}}\exp(-\norm{x}^2/2)$ is the density of the $\ell$-dimensional standard normal random variable. Furthermore, 
\begin{equation}
\sup_{A \in \Bscr(E)} |P(\Pi_E(X) \in A) - P(\Gamma_E \in A)| \leq \frac{C'}{n^{c_5}}, \label{eq:tvmain2}
\end{equation}
where $\Bscr(E)$ is the Borel $\sigma$-algebra on $E$ and $\Gamma_E$ is a standard normal random variable taking values in $E.$ 
Here $C,C'c_1,c_2,c_3,c_4,c_5>0$ are universal constants. \end{theorem}

The above theorem says that if $X$ is a random vector, then for most choices of subspaces $E$, the density of $\Pi_E(X)$ approaches the standard normal density in a `large' part of the subspace $E.$
Theorem~\ref{thm:klartag} is essentially a central limit theorem. The vector $\Pi_E(X)$ is vector of linear combinations of the components of $X$. The theorem says that these linear combinations are approximately jointly Gaussian. 
However there are important distinctions that are worth noting. First, unlike in the usual central limit theorem, \eqref{eq:tvmain} is a concrete estimate of the rate of convergence of densities. Second, notice that no assumption of independence is made on the components of $X$. And finally, note that the theorem does not guarantee that \eqref{eq:tvmain} holds for any specific subspace $E$ nor does it say it holds on the entire subspace $E$. Rather, it says that if the subspace were to be chosen uniformly at random from $G_{n,\ell}$, then \eqref{eq:tvmain} holds on a large compact set in $E$ with overwhelming probability, where this probability and the size of the compact set grows with the length of $X$.

The optimal values of the universal constants in this theorem are as yet unknown~\cite{eldan2008pointwise}. Scanning the proof from~\cite{eldan2008pointwise} reveals that the theorem holds true with $c_1=c_3=\frac{1}{100},c_2=\frac{1}{10} $ and $c_4=\frac{1}{200}$. One can obtain \eqref{eq:tvmain2}  as a corollary of \eqref{eq:tvmain}, whereby one expects $c_5 \approx c_3$ and $C' \approx C.$ For the purpose of this paper, it suffices that these are absolute constants.

The LHS of \eqref{eq:tvmain2} is $d_{TV}(\Lscr_{\Pi_E(X)},\Lscr_{\Gamma_E})$, the 
total variation distance between the laws $\Lscr_{\Pi_E(X)}$ and $\Lscr_{\Gamma_E}$ of $\Pi_E(X)$ and $\Gamma_E$, respectively. But since $d_{TV}(\Lscr_{\Pi_E(X)},\Lscr_{\Gamma_E}) = \half \norm{f_E-\varphi_\ell}_1$, it follows that $f_{E_n} \buildrel{L^1}\over \longrightarrow \varphi_\ell$, in probability. Here  the probability is understood on the underlying probability space $(\Omega,\Fscr,P)$ on which a sequence of subspaces $\{E_n\}$, $E_n:\Omega \rightarrow G_{n,\ell}$ is generated such that each $E_n$ is uniformly distributed on $G_{n,\ell}$. For obtaining the kind of results that are of interest to the team problem we study, this convergence result has to be fully exploited.

There is an additional technicality we note that helps ease the presentation in the following sections. By Lemma \ref{lem:pnl}, if $R \in V_{n,\ell}$ is chosen uniformly at random, then for any $x\in \Real^n$, 
$RR\t x$ is the projection of $x$ on a subspace of $\Real^n$ chosen uniformly at random from $G_{n,\ell}.$ Consequently, in Theorem~\ref{thm:klartag},  $\Pi_E(X) $ has the same density as $RR\t X$ for $R$ distributed uniformly on $V_{n,\ell}.$ However, by Lemma~\ref{lem:rrt}, $f_{RR\t X}(RR\t x)=f_{R\t X}(R\t x)$ for all $x\in \Real^n.$ Furthermore, $\norm{R\t x} \equiv \norm{RR\t x}$. 
Consequently, \eqref{eq:tvmain} also holds in the following version: if $X$ is isotropic with a log-concave density, $1\leq \ell \leq n^{c_1}$, and $R \in V_{n,\ell}$ is uniformly distributed on $V_{n,\ell}$ and independent of $X$, then with probability at least $1-C\exp(-n^{c_2})$, 
\[\left\lvert\frac{f_{R\t X}(x) }{\varphi_\ell(x)}-1 \right\rvert \leq \frac{C}{n^{c_3}} \quad \forall x \in \Real^\ell,\ \norm{x} \leq n^{c_4}.\]
We will use this version of \eqref{eq:tvmain} in the rest of the paper.


Our first result concerns the pointwise convergence of densities. This conclusion follows by a repeated use of Theorem~\ref{thm:klartag}. The proof is in the Appendix.
\begin{lemma} \label{lem:pointwisedensity} 
Let $\ell \in \Nbb$ and let $\{\xi_n\}$ for $n \in \Nbb, n \geq \ell^{1/c_1}$ be a sequence of isotropic random vectors, such that for each $n$, $\xi_n $ is $\Real^n$-valued and  has a log-concave density with respect to the Lebesgue measure in $\Real^n.$ Consider a sequence $\{R_n\}$ of orthogonal matrices, where each $R_n $ is chosen uniformly and independently from $V_{n,\ell}$ and independently of $\{\xi_n\}$. Let $\hat{\xi}_n=R_n\t \xi_n$ be a projection of $\xi_n$ on $\Real^\ell.$ Then $\{R_n\}$-a.s.
\[ \limn f_{\hat{\xi}_n}(x) = \varphi_\ell(x)\ \qquad \forall \ x \in \Real^\ell.\] 
\end{lemma}

We next note a consequence of Theorem~\ref{thm:klartag} that concerns the approximation of marginals (\ie, densities of linear combinations) of densities that satisfy \eqref{eq:tvmain}. Notice that this result does not follow in any obvious manner from \eqref{eq:tvmain}, since computing the marginal from the joint density would require integration over the region of $\Real^\ell$ where \eqref{eq:tvmain} does not apply. The proof is in the Appendix.
\begin{lemma} \label{lem:marginals} 
Let $n,\ell \in \Nbb.$
Let $\hat{\xi} \in \Real^\ell$ be a random vector satisfying,
\begin{equation}
\left\lvert \frac{f_{\hat{\xi}}(x)}{\varphi_\ell(x)}-1 \right\rvert \leq \frac{C}{n^{c_3}},  \label{eq:b1} 
\end{equation}
for all $x\in \Real^\ell$ with $\norm{x} \leq n^{c_4},$  
where $C,c_3,c_4$ are constants from Theorem~\ref{thm:klartag}. Let $W \in \Real^{\ell \times \ell}$ be a nonzero matrix with rank $r$ and $\zeta \sim \Nscr(0,I_\ell).$ Then, for all $x$ such that $\norm{Wx} \leq \half \sigma_{\min} n^{c_4}$
\begin{equation}
f_{W\zeta}(Wx) -f_{W\hat{\xi}}(Wx) \leq \left(\frac{C}{n^{c_3}}+\tau_{\ell,r}(n^{c_4})\right)f_{W\zeta}(Wx),  
\label{eq:marginal} 
\end{equation}
where $\sigma_{\min}$ is the smallest (positive) singular value of $W$ and 
\begin{equation}
\tau_{\ell,r}(n^{c_4})=\int_{\norm{\bar{z}}>\sqrt{\frac{3}{4}}n^{c_4}}\varphi_{\ell-r}(\bar{z})d\bar{z},\label{eq:tau} 
\end{equation}
is the weight of  tail of the $\ell-r$ dimensional Gaussian distribution ($=0$ if $\ell=r$).
\end{lemma}

Notice that although \eqref{eq:marginal} pertains to marginals, \eqref{eq:marginal} is a weaker kind of estimate than the one in \eqref{eq:tvmain}, since the LHS in \eqref{eq:marginal} \textit{does not} have the absolute value of the difference between the densities.

\section{Main results} \label{sec:main} 
We now come to the main results of this paper. 
Denote 
\[J(\gamma;\rho) := \Ebb[L(\gamma^1(H_1\rho),\hdots,\gamma^m(H_m\rho),\rho)],\]
for any function $\gamma=(\gamma^1,\hdots,\gamma^m)$ where $\gamma^i:\Real^{\ell_i} \rightarrow \Real$ and a random vector $\rho$. Here the expectation is with respect to $\rho.$

We consider an ensemble of static team problems with various distributions on the environmental randomness. A particular case is the Gaussian ensemble where the environmental randomness is a Gaussian vector for each $n$. This problem is denoted G.
 \[	\problemsmall{G}
	{\gamma}
	{\Ebb\left[ L(u,\zeta)\right] }
				 {\begin{array}{r@{\ }c@{\ }l}
			u^i&=\gamma^i(z_i),\\ 
			z_i &= H_i \zeta,
	\end{array}} \]
where $\zeta \in \Real^n, \zeta \sim \Nscr(0,I_n).$ Let $\gamma_G$ be the optimal solution to this problem. It is well known~\cite{ho1980team} that the solution to this problem is linear.

 Our main claim is that for `most' problems like NG, linear policies are approximately optimal. Let $\gamma_L^*=(\gamma_{L,1}^*,\hdots,\gamma_{L,m}^*)$ be the optimal controller for NG \textit{within} the class of linear controllers. Notice that $\gamma^*_L$ is equivalently a solution of the following problem, denoted NGL,
\begin{align*}
	\problemsmall{NGL\ }
	{\gamma_L}
	{\hspace{-2pt}\left\{\begin{array}{c}
	\half \gamma_L\t \Ebb\left[ \Hbf (\xi)Q \Hbf (\xi)\t \right] \gamma_L + \gamma_L\t \Ebb\left[ \Hbf (\xi) S \xi \right]\\ + \Ebb[q(\xi)] 	
	\end{array}\right\}}
				 {\begin{array}{r@{\ }c@{\ }l}
\gamma_L &\in& \Real^{\bar{\ell}},	
	\end{array}} 
\end{align*}
where $\Hbf(\xi) \in \Real^{\bar{\ell} \times m}$ is a block diagonal matrix formed as 
\[\Hbf(\xi) = \left [ \begin{array}{cccc}
H_1\xi & 0 & \hdots & 0 \\
0 & H_2\xi & \ddots & 0 \\
0 & 0 & \ddots & 0 \\
0 & 0 & \hdots & H_m\xi
\end{array} \right ], \]
where `$0$' stands for zero-vectors of appropriate dimensions. To see this, notice that finding linear controllers is equivalent to finding matrices $K_1,\hdots,K_m$ such that 
\[u^i = K_i\t H_i\xi.\]
But since $u^i$ is a scalar for each $i$, $K_i$ is a vector. Specifically, $K_i \in \Real^{\ell_i}$ for each $i.$ Denoting $\gamma_L= (K_1\t,\hdots,K_m\t)\t \in \Real^{\bar{\ell}}$, we get, 
\[(u^1,\hdots,u^m)\t = \gamma_L\t \Hbf(\xi).\]
From this it is easy to see that $\gamma_L^*$ solves NGL.

We first show that this problem has a unique solution.

\begin{proposition} Suppose $Q \succ 0$, $\xi$ is isotropic and $H_i$ is of full row rank for each $i=1,\hdots,m$. Then the solution to NGL is unique. 
\end{proposition}
\begin{proof}
NGL is finite dimensional optimization problem. Thus it suffices to show that it is strictly convex.
We will show that $\Ebb \left[\Hbf (\xi) Q \Hbf (\xi)\t \right]  \succ 0.$
Since $Q \succ 0,$ there exists $\lambda>0$ such that $x\t Qx \geq \lambda \norm{x}^2$ for each $x \in \Real^m.$ We have, for any $x \in \Real^m$ 
\begin{align*}
x\t \Ebb[\Hbf(\xi) Q \Hbf(\xi)\t ]x &\geq \lambda \Ebb \left[ x\t \Hbf(\xi)\Hbf(\xi)\t x\right] \\
&= \lambda \sum_{i=1}^m x_i\t H_i H_i\t x_i \geq \lambda'\norm{x}^2,
\end{align*}
for some $\lambda'>0$, since $\xi$ is isotropic and each $H_i$ is of full row rank. It follows that $\Ebb \left[\Hbf (\xi) Q \Hbf (\xi)\t \right]$ is positive definite, whereby the solution of NGL is unique.
\end{proof}

Since NGL depends only on the first and second moments of $\xi$, it follows that $\gamma_G=\gamma^*_L$.

\begin{proposition} \label{prop:gammaLG} 
Let $\gamma_L^*$ and $\gamma_G$ be the solutions of NGL and G, respectively. If $\xi\in \Real^n$ is isotropic and $\zeta \in \Real^n, \zeta \sim \Nscr(0,I_n),$  then $\gamma_L^*=\gamma_G.$ Furthermore, $J(\gamma^*_L;\xi)=J(\gamma^*_L;\zeta)=J(\gamma_G;\xi)=J(\gamma_G;\zeta).$
\end{proposition}
\begin{proof}
Since the optimal solution of G is linear, NGL and G are both finite dimensional optimization problems. The objective of NGL (of G) is a function of only the mean and the covariance matrix of $\xi$ (respectively, of $\zeta$). Since $\xi$ is isotropic, $\Ebb[\xi]=0=\Ebb[\zeta]$ and $\Cov(\xi)=I_n=\Cov(\zeta),$ whereby the result follows.
\end{proof}

Let $\zeta \sim \Nscr(0,I_n)$. The spherical symmetry of this standard normal random vector implies that for any $R \in V_{n,\ell}$, $R\t \zeta \sim \Nscr(0,I_\ell).$ Consequently, in a Gaussian ensemble of team problems, $\gamma_G$ is independent of the length of the vector of environmental randomness and of matrices $R_n \in V_{n,\ell}$.
 

\begin{proposition} \label{prop:indR} 
Let $n,\ell \in \Nbb, n\geq \ell$ and consider a problem G where $Z$ is decomposed as $Z=WR_n\t$ and $R_n\in V_{n,\ell}$. Then for fixed $Q,W$ the solution $\gamma_G$ is independent of $R_n$ and $n.$ 
\end{proposition}
\begin{proof}
G is determined by the distribution of $Z\zeta=W R_n\t \zeta.$ Since, $W$ is fixed, this further depends only on the distribution of $R_n\t \zeta$, which is distributed as $\Nscr(0,I_\ell)$ for any $R_n$ and any $n$. The result follows.
\end{proof}

Since $\gamma_G$ is independent of $R_n$ and $n$, when referring to problem to G we may equivalently take 
$R_n=I_\ell$ and hence $W=Z$. 

\subsection{Tight error bound} \label{sec:error} 

In this section we derive an asymptotically tight lower bound on $J(\gamma_n^*;\xi_n)$ which shows that $J(\gamma_n^*;\xi_n) \rightarrow J(\gamma_G;\zeta)$ almost surely. 
Denote, \[J(\gamma;\rho|A) := \Ebb[\I{A} L(\gamma^1(H_1\rho),\hdots, \gamma^m(H_m\rho),\rho)], \]
where $\I{A}$ is short hand for $\I{\rho \in A}$, which is a random variable that is unity  if $\rho \in A$ and $0$ otherwise. 
Let, $$A_n:= \{x \in \Real^\ell | \norm{x} \leq n^{c_4}\}.$$ Let $\zeta \sim \Nscr(0,I_\ell).$ The crux of this bound lies in completely exploiting that,
\[J(\gamma_G;\zeta) \geq J(\gamma_n^*;\xi_n ), \quad \forall n,\]
whereby $J(\gamma_G;\zeta)$ is a \textit{deterministic} and \textit{uniform} (with $n$) upper bound on $J(\gamma_n^*;\xi_n).$ To this end, define,
\begin{equation}
v_n :=\min_\gamma J(\gamma;\zeta|A_n), \qquad \gamma_n \in \arg \min_\gamma J(\gamma; \zeta|A_n). \label{eq:vndef} 
\end{equation}
By the same argument as in Proposition~\ref{prop:indR}, $v_n$ is also independent of $R \in V_{\ell,\ell}$. Hence we take $R=I_\ell$ when referring to $v_n$, whereby the argument of $\gamma_n^i$ is $H_i\zeta$~\footnote{As an alternative one may explicitly include $R \in V_{\ell,\ell}$ into the problem formulation and consider the quantity $\tilde{\gamma}_n \in \arg\min_\gamma J(\gamma;\tilde{\zeta}|A_n)$ where $\tilde{\zeta}:=R\t \zeta$. In this situation, the argument of $\tilde{\gamma}_n^i$ is $W_i\tilde{\zeta}$ where $W_i$ satisfies $H_i=W_iR\t$. But since $\zeta $ and $\tilde{\zeta}$ are both distributed as $\Nscr(0,I_\ell)$, we have $J(\gamma;\zeta|A_n)=J(\gamma;\tilde{\zeta}|A_n)$ and there is no loss of generality in considering $R=I_\ell$.}. 
We denote $$u_n^i := \gamma_n^i(H_i \zeta) \aur u_n := (u_n^1,\hdots,u_n^m).$$ $u_n^i$ is the control \textit{action} of player $i$ whereas $\gamma_n^i$ is the policy. 
$u_n^i$ is itself a random variable and is measurable with respect to the observations of player $i.$ Let $\Gscr_i$ be the $\sigma$-algebra generated by $H_i\zeta.$ Then for each $n,$ $u^i_n$ is a $\Gscr_i$-measurable random variable, for $i=1,\hdots,m$. 

While considering problem NG, we again consider an ensemble of static team problems with parameters $Q,W,\{R_n\},\{\xi_n\},$ where $\xi_n \in \Real^n$ is an isotropic random vector with log-concave density, $R_n \in V_{n,\ell}$ a uniformly random orthonormal matrix, independent of $\{\xi_k\}$ and $\gamma_n^*$, the solution of the $n\th$ problem in this ensemble.

We first show a fundamental error bound. Constants $C,c_1,c_2,c_3,c_4>0$ are from Theorem~\ref{thm:klartag}. 
\begin{proposition}[Error bound] Let $n,\ell \in \Nbb$ such that $n\geq \ell^{1/c_1}$ and $\frac{C}{n^{c_3}}<1$,  $\zeta \sim \Nscr(0,I_\ell)$ and consider an ensemble of static team problems as above.  
With probability at least $1-C\exp(-n^{c_2}) $, the ($R_n$-dependent quantity) $J(\gamma_n^*;\xi_n)$ satisfies (the deterministic bounds), \label{prop:bounds} 
\begin{equation}
J(\gamma_G;\zeta)\geq  J(\gamma^*_n;\xi_n) \geq v_n - \frac{\frac{C}{n^{c_3}}}{1-\frac{C}{n^{c_3}}}J(\gamma_G;\zeta).\label{eq:bounds} 
\end{equation}
\end{proposition} 
\begin{proof}
The upper bound follows from Propositions~\ref{prop:gammaLG} and~\ref{prop:indR}. 
For the lower bound, by Theorem~\ref{thm:klartag}, with probability at least $1-C\exp(-n^{c_2}) $,
\begin{equation}
f_{\hat{\xi}}(x) \geq \left (1- \frac{C}{n^{c_3}}\right ) f_{\zeta}(x), \label{eq:den1} 
\end{equation}
for all $x \in A_n$. Therefore using \eqref{eq:den1} (since $L\geq 0$), 
\begin{align*}
J(\gamma^*_n;\xi_n)  &\geq J(\gamma^*_n;\xi_n|A_n) \geq J(\gamma_n^*;\zeta|A_n) -\frac{C}{n^{c_3}}J(\gamma_n^*;\zeta|A_n), \label{eq:jgamma*} 
\end{align*}
By \eqref{eq:den1}, since $1-\frac{C}{n^{c_3}}$ is assumed positive, we get
\begin{align*}
J(\gamma_n^*;\xi_n) &\geq J(\gamma_n^*;\zeta|A_n) - \frac{\frac{C}{n^{c_3}}}{1-\frac{C}{n^{c_3}}} J(\gamma_n^*;\xi|A_n), \non\\
 & \geq v_n - \frac{\frac{C}{n^{c_3}}}{1-\frac{C}{n^{c_3}}} J(\gamma_G;\zeta), \non
\end{align*}
by definition of $v_n$ in \eqref{eq:vndef} and once more by Proposition~\ref{prop:gammaLG} and~\ref{prop:indR}.
\end{proof}

Consequent to Proposition~\ref{prop:bounds}, our aim for the rest of this section is to show that $v_n\rightarrow J(\gamma_G;\zeta)$. This will establish that the bounds in Proposition~\ref{prop:bounds} are asymptotically tight. But before we get on with this task, we note a meta-converse to the statement that the optimal controller for the Gaussian problem is independent of $R_n $ and $n$ (cf. Proposition~\ref{prop:indR}).

\begin{proposition}
Suppose we have that $\{R_n\}$-a.s., 
$\gamma^*_n$ is independent of $n$ and equal to $\gamma^*$ (say) for all $n$. Then $\gamma^*=\gamma_G$, $\{R_n\}$-a.s. \label{prop:indep} 
\end{proposition}
\begin{proof}
Let $\ell' \geq \ell$ and $\zeta \sim \Nscr(0,I_\ell).$ 
Arguing as in the proof of Lemma~\ref{prop:bounds}, we have that for all $n \geq \ell'^{1/c_1}$, with probability at least $1-C\exp(-n^{c_2}),$
\[J(\gamma_G;\zeta)\geq J(\gamma^*;\xi_n) \geq J(\gamma^*;\zeta|A_n) - \frac{\frac{C}{n^{c_3}}}{1-\frac{C}{n^{c_3}}}J(\gamma_G;\zeta).\]
Arguing as in the proof of Lemma~\ref{lem:pointwisedensity}, we get that the event 
\[J(\gamma_G;\zeta) \geq \limn J(\gamma^*;\xi_n) \geq \limn J(\gamma^*;\zeta|A_n) = J(\gamma^*;\zeta),\]
holds with probability at least $1-K\exp(-(\ell')^{c_2/c_1})$ (the last equality is from monotone convergence theorem).  Since $\ell'\geq \ell$ is arbitrary, we get that $J(\gamma_G;\zeta) \geq J(\gamma^*;\zeta),$ $\{R_n\}$-a.s. But since $\gamma_G$ is the unique solution of problem $G$, it follows that $\gamma^*=\gamma_G,$ $\{R_n\}$-a.s.
\end{proof}

We now come to showing $v_n \buildrel{n}\over \longrightarrow J(\gamma_G;\zeta).$ Our first observation is that $\{v_n\}$ is in fact convergent.

\begin{lemma} For all $n$, \label{lem:vncon} 
$J(\gamma_G;\zeta) \geq v_n \geq v_{n-1}. $
Hence 
$\left\{v_n\right\} $ is convergent. 
\end{lemma}
\begin{proof}
For each $\gamma,$ since $L(\cdot,\cdot) \geq 0$,
\begin{align*}
L(\gamma(\Hbf(\zeta)),\zeta)&\geq \I{A_n}L(\gamma(\Hbf(\zeta)),\zeta) \geq \I{A_{n-1}}L(\gamma(\Hbf(\zeta)),\zeta). 
\end{align*}
Taking expectations and minimizing each term with respect $\gamma$ gives that $J(\gamma_G;\zeta) \geq v_n \geq v_{n-1}$. Since $\{v_n\}$ is an increasing bounded sequence, it is convergent.
\end{proof}

With a slight abuse of notation, we write $J(u;\zeta)$ and $J(u;\zeta|A_n)$ to mean $J(\gamma;\zeta)$ and $J(\gamma;\zeta|A_n)$, respectively, with $u^i=\gamma^i(H_i\zeta)$ for all $i \in \Nscr.$

\begin{lemma} There is a subsequence $\{\I{A_{n_k}}u_{n_k}\}_k$ of $\{\I{A_n}u_n\}_n$ such that 
$\{\I{A_{n_k}}u_{n_k}^i\}_{k}$ converges weakly in $L^2(\Omega,\Fscr,P)$ for each $i=1,\hdots,m$. Here $\I{A_n} =\I{\zeta \in A_n}$, where $\zeta \sim \Nscr(0,I_\ell).$ \label{lem:subseq} 
\end{lemma}
\begin{proof}
Notice that $J(u;\zeta|A_n)=J(\I{A_n}u;\zeta) - \Ebb[\I{A_n^c}q(\zeta)].$
Since $v_n \leq J(\gamma_G;\zeta)$, we have that 
\[\sum_{i =1}^m \Ebb[\lambda_{\min} (\I{A_n}u_n^i)^2 + \I{A_n}u_n^i S_i\t\zeta ] \leq J(\gamma_G;\zeta) \] 
for all $n$ where $\lambda_{\min}>0$ is a smallest eigenvalue of $Q$. Consequently, for each $=1,\hdots,m,$ $\{\I{A_n} u_n^i\}_n$ lies in a bounded subset of $L^2(\Omega,\Fscr,P).$ Thus, for $i=1$, there exists a subsequence that converges weakly. Taking a further subsequence, we find there is a subsequence that converges weakly for both $i=1 $ and $2$. Thus taking subsequences of subsequences $m$ times, we get that there is a subsequence of the original sequence that converges weakly for each $i=1,\hdots,m.$
\end{proof}

Let the weak limit of $\{\I{A_{n_k}}u_{n_k}^i\}$ be $\bar{u}^i.$ A crucial ingredient in the final proof is showing that the random variable $\bar{u}^i$ indeed a feasible action for the static team problem, \ie, showing that it is 
measurable with respect to $\Gscr_i.$
For showing this, we need a celebrated theorem from functional analysis due to Banach, Saks and Mazur. The version we need can be found in~\cite[Thm 5.13-1(c)]{ciarlet2013linear}.
\begin{theorem}[Banach-Saks-Mazur] 
Let $X$ be a normed linear space and let $\{x_k\}_k$ be a sequence in $X$ such that $x_k \rightarrow x$ weakly. Then, for each $n \in \Nbb$ there exists $n_0(n) \in \Nbb$ and scalars $\mu^n_k \geq 0, n \leq k\leq n+n_0(n)$ with $\sum_{k=n}^{n+n_0(n)} \mu^n_k=1$ such that 
\[z_n:= \sum_{k=n}^{n+n_0(n)} \mu^n_k x_k \rightarrow x,\]
strongly in $X$ as $n \rightarrow \infty.$
\end{theorem}
We now show the following lemma, which will be used to establish that $\bar{u}^i$ is $\Gscr_i$-measurable.
\begin{lemma} 
Let $(\Omega,\Fscr,P)$ be a probability space and let $\{f_n\}_n$ be a sequence in $L^2(\Omega,\Fscr,P)$. Suppose $\Gscr$ is a sub-$\sigma$-algebra of $\Fscr$ such that 
$f_n \in \sigma(\Gscr)$ for all $n$. Further, let $g_n =\I{A_n}$ where $A_n \uparrow \Omega$ pointwise as $n \rightarrow \infty$. If $f_ng_n \rightarrow h$ weakly in $L^2(\Omega,\Fscr,P),$ then  $h \in \sigma(\Gscr)$. \label{lem:phoemuex} 
\end{lemma}
\begin{proof}
Let $h_n=f_ng_n$. Since $\{h_n\}$ converges weakly, by the Banach-Saks-Mazur theorem, there exists a sequence $\{n_0(n)\}_{n \in \Nbb}$ and 
 (double) sequence of nonnegative scalars $\{\mu^n_k\}_{n \in \Nbb, n \leq k \leq n+n_0(n)}$ such that for each $n\in \Nbb,$ $\sum_{k=n}^{n+n_0(n)} \mu^n_k=1 $ and the sequence $\{\kappa_n\}$ where,
 \[\kappa_n=\sum_{k=n}^{n+n_0(n)} \mu^n_k f_kg_k,\]
 converges strongly in $L^2(\Omega,\Fscr,P)$ to $h$. Thus there exists a subsequence $\{\kappa_{n_r}\}_r$ such that $\kappa_{n_r} \buildrel r \over \rightarrow h$ almost surely. 
 For (almost every) $x \in \Omega$, there exists $r_0 \in \Nbb$ such that $g_{n_r}(x)=1$ for all $r \geq r_0.$ Thus for almost every $x$, 
\begin{equation}
h(x)=\lim_{r\rightarrow \infty} \kappa_{n_r}(x)= \lim_{r\rightarrow \infty} \sum_{k=n_r}^{n_r+n_0(n_r)} \mu^{n_r}_k f_k(x).
\end{equation}
The right hand side is the limit of a sequence of functions in $\sigma(\Gscr)$. It follows that $h \in \sigma(\Gscr).$
\end{proof}

We now have all pieces in place to show $v_n \rightarrow J(\gamma_G;\zeta).$
\begin{theorem}
Let $v_n$ be as defined in \eqref{eq:vndef} and $\zeta \sim \Nscr(0,I_\ell)$. Then, 
$v_n \rightarrow J(\gamma_G;\zeta)$ as $n\rightarrow \infty.$
\end{theorem}
\begin{proof}
Let $u_n^i = \gamma_n^i(H_i \zeta).$ By Lemma~\ref{lem:subseq}, there exists a subsequence $\{u_{n_k}\}_k$ such that $u_{n_k}^i$ converges weakly for each $i=1,\hdots,m.$ Let this weak limit be $\bar{u}=(\bar{u}^1,\hdots,\bar{u}^m)$; this weak limit also lies in $L^2(\Omega,\Fscr,P).$ By Lemma~\ref{lem:vncon}, $\{v_n\}$ is convergent, whereby each subsequence of $\{v_n\}$ converges to the same limit. Thus it suffices to show that $\limk v_{n_k} =J(\gamma_G;\zeta).$ 


Now by convexity of $L$, we have for all $n$,
\begin{align}
v_n&\geq J(\bar{u};\zeta|A_n) + \Ebb[\I{A_n}(Q\bar{u}+S\zeta)\t(u_n -\bar{u})] \non\\
&= J(\bar{u};\zeta|A_n) + \Ebb [(Q\bar{u}+S\zeta)\t(\I{A_n}u_n -\bar{u}) ] \non \\ 
&\qquad+\Ebb[ (Q\bar{u}+S\zeta)\t(\bar{u}-\I{A_n}\bar{u}) ]\non \\ 
&\geq J(\bar{u};\zeta|A_n) + \Ebb[(Q\bar{u}+S\zeta)\t(\I{A_n}u_n -\bar{u})] \non \\ 
&\quad - \Ebb \left[\norm{Q\bar{u}+S\zeta}^2\right]^{1/2}  \Ebb\left[\norm{\I{A_n}\bar{u}-\bar{u}}^2 \right]^{1/2}, \label{eq:finalarg} 
\end{align}
where the last inequality follows from Cauchy-Schwartz inequality. 
Now, take limits on both sides along the above subsequence $\{n_k\}_k.$ The second term in the RHS of \eqref{eq:finalarg} is $\sum_{i=1}^m \Ebb[(Q\bar{u}+S\zeta)_i (\I{A_n}u_n^i -\bar{u}^i)].$
Since $\bar{u}^i$ is the weak limit of $\{\I{A_n}u_{n_k}^i\}_k$ for each $i$, this term vanishes as $k \rightarrow \infty.$ Since $\bar{u} \in L^2(\Omega,\Fscr,P),$ the third term in \eqref{eq:finalarg} is finite and by the dominated convergence theorem, this term vanishes as $k \rightarrow \infty.$ And by the monotone convergence theorem, 
$J(\bar{u};\zeta|A_{n_k}) \buildrel k\over \rightarrow J(\bar{u};\zeta).$ Consequently, 
\[\limn v_n = \limk v_{n_k} = J(\bar{u};\zeta).\]
But by Lemma~\ref{lem:phoemuex}, for each $i =1,\hdots,m,$ 
$\bar{u}^i$ is a measurable function of $H_i\zeta$ and is thus a feasible control action for problem G. Hence, $J(\bar{u};\zeta) \geq J(\gamma_G;\zeta).$ Finally, using Lemma~\ref{lem:vncon}, 
we get 
\[J(\gamma_G;\zeta) \geq \limn v_n \geq J(\bar{u};\zeta) \geq J(\gamma_G;\zeta),\]
whereby the result follows.
\end{proof}

We thus have our final theorem.
\begin{theorem}
Let $\ell \in \Nbb$ and let $\{\xi_n\}$ for $n \in \Nbb, n \geq \ell^{1/c_1}$ be a sequence of isotropic random vectors, such that for each $n$, $\xi_n $ is $\Real^n$-valued and  has a log-concave density with respect to the Lebesgue measure in $\Real^n.$ Let $\zeta \sim \Nscr(0,I_\ell)$ and consider an ensemble of static team problems with parameters $Q,W,\{R_n\},\{\xi_n\}$, where each $R_n \in \Real^{n \times \ell}$ and is chosen uniformly and independently from $V_{n,\ell}$ and independently of $\{\xi_k\}.$ Then, 
$\{R_n\}$-a.s., we have
\begin{equation}
\limn J(\gamma^*_n;\xi_n) = J(\gamma_G;\zeta).\label{eq:final} 
\end{equation}
\end{theorem}
\begin{proof}
By Proposition~\ref{prop:bounds}, \eqref{eq:bounds} holds with $\{R_n\}$-probability at least $1-C\exp(-n^{c_2}).$ Let $\ell'\geq \ell.$ 
Arguing as in Proposition~\ref{prop:indep}, with probability at least $1-K\exp(-(\ell')^{c_2/c_1}),$ \eqref{eq:final} holds. Since $\ell'$ is arbitrary, \eqref{eq:final} holds $\{R_n\}$-a.s.
\end{proof}

One may note that all assumptions of the LQG setup have been utilized in the proof. The linearity of the information and quadratic cost was used in showing $\gamma_L^*=\gamma_G,$ and the spherical symmetry of the standard normal distribution gave that the Gaussian cost did not depend on $n$. This yielded the crucial deterministic upper bound of $J(\gamma_G;\zeta)$ on $J(\gamma^*_n;\xi_n)$. The convexity of the cost function was used in lower bounding $v_n$, following which the quadratic nature of the cost was critical for the weak convergence-based argument. Finally, thanks to the static information structure, each term in $\{u_{n}^i\} $ was measurable with respect to $\Gscr_i$, which did not vary with $n$, thereby allowing us to apply Lemma~\ref{lem:phoemuex}. Indeed, looking back one finds that the delicate agreement of all assumptions is a fascinating and recurring theme in LQG theory. But this complete exploitation of assumptions also makes avenues for generalization of these results harder to find.

\subsection{An explicit bound}
The bound in Proposition~\ref{prop:bounds}, while tight is somewhat less explicit. 
We now show another bound that is explicit, but whose tightness we are not able to establish unconditionally. 

Suppose $Z$ is written as $Z=W R\t$.
Below and later in this paper we note that if $\rho \in \Real^n$ is a random vector,
\begin{align*}
J(\gamma;\rho) &= \int_{\Real^n} L(\gamma(\Hbf (x)),x)f_\rho(x)dx \\
&= \int_{\Real^\ell} L(\gamma(W z),Wz)f_{R\t \rho}(z)dz, \\
&=\int_{\range(W)}L(\gamma(w),w)f_{WR\t \rho}(w)dw,
\end{align*}
where $\gamma(\Hbf(x))\equiv (\gamma^1(H_1x),\hdots,\gamma^m(H_mx)).$ Further, since $Z=WR\t$, it follows that $\gamma$ is then a function of $W R\t\rho$ and the cost function $L(\gamma,x)\equiv L(\gamma,Wz)$ (by abuse of notation) where $z=R\t x.$ Thus, $J(\gamma;\rho)$ can be written in terms of the density of $R\t \rho$. Further, since the value of $\gamma$ and $L$ are determined by the value of $Wz$, 
$J(\gamma;\rho)$ can be written in terms of the density of $R\t \rho$ or of $WR\t \rho.$ This latter density is with respect the Lebesgue measure on the range of $W$, denoted `$dw$'. 
Similarly, $q(x)$ can also be written (by abuse of notation) as $q(Wz).$

Moreover, write $W\t = \left[ 
W_0\t ; 
W_1\t; \hdots W_m\t
 \right]  $ such that $S\t = RW_0\t $ and $
H_i\t =  RW_i\t, i\in \Nscr.  $ Similarly, let $w\in \range(W)$ be written as $w=(w_0,w_1,\hdots,w_m)$ where $w_i \in \range(W_i), i=0,1,\hdots,m.$ Recall that any controller $\gamma^i(w)$ is a function of $w_i$ alone. The proof of the theorem below is relegated to the Appendix. 
\begin{theorem}[Explicit bound] \label{thm:error} 
Let $\ell,n \in \Nbb,$ $n\geq \ell^{1/c_1}.$
Consider a static team problem with parameters $Q,W,R_n,\xi_n$ where $R_n \in V_{n,\ell}$ is uniformly distributed, and the environmental randomness $\xi_n \in \Real^n$ is isotropic with a log-concave density. Let $\gamma^*_n$ (with values viewed as column vectors in $\Real^m$) be an optimal controller for this problem and suppose the terms in matrices $\Ebb[\gamma^*_n(\Hbf(\zeta)) \gamma^*_n(\Hbf(\zeta))\t]$ and $\Ebb[\gamma^*_n(\Hbf(\zeta))\zeta\t]$, are finite where 
$\zeta \sim \Nscr(0,I_\ell)$.
Then, if $W$ is of rank $r$ and $n$ is such that $1- \frac{C}{n^3}-\tau_{\ell,r}(n^{c_4}) >0,$ we have with probability at least $1-C\exp(-n^{c_2}),$
\begin{align}
0\leq J(\gamma^*_L;\xi_n) - J(\gamma^*_n;\xi_n) &\leq \frac{\frac{C}{n^{c_3}}+\tau_{\ell,r}(n^{c_4})}{\left(1- \frac{C}{n^3}-\tau_{\ell,r}(n^{c_4})\right)}J(\gamma^*_G;\zeta)\non \\ 
+& \int_{A^c}\hspace{-1mm} L(\bar{\gamma}(w),w)f_{W\zeta}(w)dw.
\label{eq:error} 
\end{align}
Here $$A=\left \{Wx \left\lvert {\ds \max_{i=0,\hdots,m}}\norm{W_ix} \leq\frac{1}{\sqrt{m+1}} \half \sigma_{\min} n^{c_4}, x \in \Real^\ell \right. \right \}, $$ $C,c_1,c_2,c_3,c_4>0$ are absolute constants from Theorem~\ref{thm:klartag},  $\sigma_{\min}$ is the smallest positive singular value of $W$, $\tau$ is as defined in \eqref{eq:tau}, and for $i=1,\hdots,m,$ 
\[\bar{\gamma}^{i}(w_i) = \begin{cases}(\gamma^*_n)^i(w_i) & \eef \norm{w_i} \leq \frac{1}{\sqrt{m+1} }\half \sigma_{\min} n^{c_4}, \\
0 &\ow. 
\end{cases} \]
\end{theorem}

Note that the RHS of \eqref{eq:error} is small for large $n$ if and only if the second term therein is small. This is ensured by the assumptions of the finiteness of $\Ebb[\gamma^*_n(\Hbf(\zeta)) \gamma^*_n(\Hbf(\zeta))\t]$ and $\Ebb[\gamma^*_n(\Hbf(\zeta))\zeta\t]$, whereby $J(\bar{\gamma};\zeta) <\infty.$ 
In that case, Theorem \ref{thm:error} says that the ratio of $J(\gamma^*_L;\xi_n)$ and $J(\gamma^*_n;\xi_n)$ is close to unity for `most' problems. 
If the second term in \eqref{eq:error} is large,  then \eqref{eq:error} is a weak bound since we know $J(\gamma^*_L;\xi_n) - J(\gamma^*_n;\xi_n) $ admits the (uniform) finite upper bounds 
$J(\gamma^*_L;\xi_n) - J(\gamma^*_n;\xi_n)  \leq J(\gamma^*_G;\zeta)\leq J(0;\zeta) =\Ebb[q(\zeta)]$.

\subsection{Pointwise convergence to a linear solution}

In our final result we establish a different mode of convergence. We show that if the sequence of solutions of an ensemble of team problems NG converges pointwise almost surely, then its limit is $\gamma_G$ almost surely (the probability referred to here is with respect to the distribution of $\{R_n\}$). 
\begin{theorem}
Let $\ell \in \Nbb$ and let $\{\xi_n\}$ for $n \in \Nbb, n \geq \ell^{1/c_1}$ be a sequence of isotropic random vectors, such that for each $n$, $\xi_n $ is $\Real^n$-valued and  has a log-concave density with respect to the Lebesgue measure in $\Real^n.$ Consider an ensemble of static team problems with parameters $Q,W,\{R_n\},\{\xi_n\}$, where each $R_n \in \Real^{n \times \ell}$ and is chosen uniformly and independently from $V_{n,\ell}$ and independent of $\{\xi_k\}.$ Let $\gamma^*_n$ be a solution of the $n\th$ problem in this ensemble. Suppose 
$\hat{\gamma}:=\lim_{n \rightarrow \infty} \gamma^*_n$ (pointwise) exists $\{R_n\}$-a.s. Then $\hat{\gamma}=\gamma_G$, $\{R_n\}$-a.s. 
\end{theorem}
\begin{proof}
Let $\zeta \sim \Nscr(0,I_\ell)$ and $\hat{\xi}_n=R_n\t \xi_n.$
We have,
\begin{align*}
J(\hat{\gamma};\zeta) &= \int_{\Real^\ell} L(\hat{\gamma}(Wx),Wx)\varphi_\ell(x) dx \\ 
&\buildrel{(a)}\over= \int_{\Real^\ell}\liminf_{n \rightarrow \infty} L(\gamma_n^*(Wx),Wx)f_{\hat{\xi}_n}(x)dx 
\end{align*}
with probability 1. Here $(a)$ follows from the pointwise convergence showed in Lemma~\ref{lem:pointwisedensity} and the definition of $\hat{\gamma}.$ Thus, if $\zeta_n\sim \Nscr(0,I_n)$
\begin{align*}
J(\hat{\gamma};\zeta)&\buildrel{(b)}\over\leq \liminf_{n \rightarrow \infty}\int_{\Real^\ell} L(\gamma^*_n(Wx),Wx) f_{\hat{\xi}_n}(x) dx,\\
&=\liminf_{n \rightarrow \infty}J(\gamma_n^*;\xi_n), \\
& \buildrel{(c)}\over \leq \liminf_{n \rightarrow \infty} J(\gamma_G;\xi_n),\\
&\buildrel{(d)}\over=\liminf_{n \rightarrow \infty}J(\gamma_G;\zeta_n), \\ 
&\buildrel{(e)}\over= J(\gamma_G;\zeta),
\end{align*}
$(b)$ follows from Fatou's lemma, $(c)$ from the optimality of $\gamma^*_n$, $(d)$ from Proposition~\ref{prop:gammaLG} and $(e)$ from Proposition~\ref{prop:indR}. 
But since $\gamma_G$ is the unique solution of G, it follows that $\hat{\gamma}=\gamma_G,$ as required.
\end{proof}

\section{Some uniform convergence results} \label{sec:uniform} 
We end this paper by noting a few additional results. These results do not pertain to the team problem per se, but would perhaps be useful to other researchers working on similar themes. Hence we include them in this paper.
\subsection{Uniform convergence of tails}
Our first result is a uniform integrability-type result.
\begin{lemma} Let $\ell\in \Nbb$ and consider an ensemble of static team problems with parameters $Q,W,\{R_n\},\{\xi_n\}$
where $R_n \in V_{n,\ell}$ is uniformly distributed and the environmental randomness is given by a sequence of isotropic random vectors $\{\xi_n\}$ each with a log-concave density and independent of $R_n$. Let  $\hat{\xi}_n=R_n\t \xi_n$ and denote,
\[T(n,k):=\int_{x: \norm{x}>k}L(\gamma_n^*(Wx),Wx) f_{\hat{\xi}_n}(x)dx, \]
where $\gamma_n^*$ is the optimal controller for the $n\th$ problem in the ensemble. Let 
\[\delta_{n,k}:= T(n,k)-T(n,k-1),\]
and suppose there exists a function $g:\Nbb \rightarrow \Real $ such that $\delta_{n,k}\leq g(k) $ and $\sum_k g(k) <\infty.$ 
Then, for all $\epsilon>0$ there exists $k_0$ such that
\[\limn T(n,k)<\epsilon, \quad \forall k \geq k_0. \]
\end{lemma}
\begin{proof}
It suffices to show that 
\[\limk \limn T(n,k)=0.\]
Since $J(\gamma_n^*;\xi_n) \leq J(\gamma_G;\zeta)<\infty,$  $\{R_n\}$-a.s., 
\[\limk T(n,k)=0, \quad \forall n,\]
whereby $\limn \limk T(n,k)=0.$
Consider $\Nbb$ endowed with counting measure $\nu$. Since $g \in L^1(\Nbb,\nu)$, the dominated convergence theorem gives
\[\int_\Nbb \limn \delta_{n,k}d\nu(k)=\limn \int_\Nbb \delta_{n,k}d\nu(k).\]
By definition,
\[ \limk \sum_{t=1}^k \limn \delta_{n,t}=\limn \limk \sum_{t=1}^k \delta_{n,t},\]
which simplifies to 
\[\limk \limn T(n,k) =\limn \limk T(n,k)=0,\]
as required.
\end{proof}

\subsection{Uniform convergence of densities}
Lemma~\ref{lem:pointwisedensity} showed the pointwise convergence of densities of projections of random vectors with log-concave densities. In this section, with additional assumptions we show the uniform convergence of such densities.

If $f(x)\equiv e^{-G(x)}$, where $G: \Real^n \rightarrow \Real$ is a convex function,  is a log-concave density, $G$ must necessarily approach infinity as $\norm{x} \rightarrow \infty.$ This implies that $G$ cannot grow sublinearly, for otherwise it would cease to be convex. Formalizing this observation leads to the following pointwise estimate from~\cite{cule2010theoretical}.

\begin{lemma}~\cite[Lemma 1]{cule2010theoretical}
Let $f$ be a log-concave density on $\Real^d$. Then there exist $a=a(f)>0$ and $b=b(f)\in \Real$ such that $f(x) \leq \exp(-a\norm{x}+b)$ for all $x\in \Real^d.$ \label{lem:cule} 
\end{lemma}

Let $\xi \in \Real^n$ have a log-concave density. Since log-concavity is preserved under affine transformations (see, \eg,~\cite{dharmadhikari1988unimodality}), the density of any projection, $\hat{\xi}=R\t \xi$, where $R \in V_{n,\ell}$  is also log-concave. Our interest is in deriving a \textit{uniform} estimate for the density $R\t \xi$ over all $R\in V_{n,\ell}$ \textit{and} over all $n$ for fixed $\ell$. The following lemma obtains this. Below, $\Gamma(\cdot)$ is the Gamma function.

\begin{lemma} \label{lem:subexp}  Suppose $\xi \in \Real^n$ is a random vector with log-concave density satisfying 
$$f_{\xi}(x) \leq \exp(-a\sqrt{2}\norm{x}+b_n) \quad \forall x\in \Real^n,$$ for some constants $a>0$ 
and $b_n$ satisfying,
\begin{equation}
\exp b_n \leq \exp b \cdot\left( \frac{a}{\sqrt{\pi}}\right)^{n-\ell} \frac{\Gamma((n-\ell)/2)}{2\Gamma(n-\ell)}, \label{eq:bn} 
\end{equation}
for some constant $b.$ Let $R \in V_{n,\ell} $  and let $\hat{\xi}=R\t \xi.$ Then, 
\begin{equation}
f_{\hat{\xi}}(x) \leq \exp\left(-a\norm{x}+b\right), \quad \forall x \in \Real^\ell.\label{eq:expbound} 
\end{equation}
\end{lemma}
\begin{proof}
Let $P$ be an orthogonal matrix given by $P=[R\ \bar{R}]$ where $\bar{R} \in \Real^{n\times n-\ell}$ is an orthonormal matrix with columns orthogonal to $R$.  Thus $f_{\hat{\xi}}$ is the marginal density of the first $\ell$ components of $P\t\xi$. By the change of variables formula,
\begin{align*}
&f_{\hat{\xi}}(x)
=\int_{\bar{x} \in \Real^{n-\ell}} f_{P\t \xi}(x,\bar{x})d\bar{x}, \\
&= \int_{\bar{x} \in \Real^{n-\ell}}f_\xi(Rx+ \bar{R}\xbar)d\bar{x},\\
&\leq \int_{\Real^{n-\ell}} \exp(-a\sqrt{2}\sqrt{\norm{x}^2+\norm{\bar{x}}^2}+b_n)d\xbar, \\
&\buildrel{(a)}\over=\frac{2\pi^{(n-\ell)/2}}{\Gamma((n-\ell)/2)}\hspace{-1mm}\int_{r=0}^\infty \exp(-a\sqrt{2}\sqrt{\norm{x}^2+r^2}+b_n)r^{n-\ell-1}dr, \\
&\buildrel{(b)}\over\leq \frac{2\pi^{(n-\ell)/2}\exp(-a\norm{x}+b_n)}{\Gamma((n-\ell)/2)} \int_0^\infty \exp\left(-ar\right)r^{n-\ell-1}dr, \\
&\buildrel{(c)}\over=\frac{2(2\pi)^{(n-\ell)/2}\exp(-a\norm{x}+b_n)}{\Gamma((n-\ell)/2)} \frac{\Gamma(n-\ell)}{a^{n-\ell}},
\end{align*}
where $(a)$ follows from the surface area of the sphere in $\Real^{n-\ell}$ of radius $r$, $(b)$ uses that $\sqrt{p^2+q^2} \geq \frac{1}{\sqrt{2}}(p+q)$ for any $p,q>0$ (Jensen's inequality),    and $(c)$ follows from the definition of the Gamma function. Using \eqref{eq:bn}, the result follows.
\end{proof}

The proof of the above theorem reveals why $b_n$ must grow with $n$ even while $a$ can be constant. $b_n$ is the scaling that ensures that $f_{\xi_n}$ is a density for each $n.$ As $n$ grows, the scaling required changes, since the volume of  the $n$-dimensional ball changes. 

As a consequence of Lemma~\ref{lem:subexp}, 
we get that any $\ell$ dimensional projection of a sequence of log-concave distributed random vectors of increasing length satisfies the estimate \eqref{eq:expbound} provided the densities of the vectors satisfy a bound given by \eqref{eq:bn}. Theorem~\ref{thm:uniform} below shows that this implies that the densities of the projections converge uniformly to the standard Gaussian density.

\begin{theorem}  \label{thm:uniform}
Let $\ell \in \Nbb$ and let $\{\xi_n\}$ for $n \in \Nbb, n \geq \ell^{1/c_1}$ be a sequence of isotropic random vectors, such that for each $n$, $\xi_n $ is $\Real^n$-valued and  has a log-concave density with respect to the Lebesgue measure in $\Real^n.$ Further assume that there exist constants $a>0$ and a sequence $\{b_n\}$ satisfying \eqref{eq:bn} for some $b \in \Real$ such that 
\[f_{\xi_n}(x) \leq \exp(-a\norm{x}+b_n), \qquad \forall x \in \Real^n,\]
and all $n.$ 
Consider a sequence $\{R_n\}$ of orthogonal matrices, where each $R_n \in \Real^{n \times \ell}$ and is chosen uniformly and independently from $V_{n,\ell}$ and  let $\hat{\xi}_n=R_n\t \xi_n$ Then, 
\[\sup_{x\in \Real^\ell}| f_{\hat{\xi}_n}(x) -\varphi_\ell(x)| \leq \max \left\{ \frac{K'}{n^{c_3}}, \exp(-a'n^{c_4}+b')\right\},\]
with probability at least $1-C\exp(-n^{c_2})$ where $K'$ is a constant depending only on $\ell$ and $a',b'$ depend only $a,b,\ell$. 
Furthermore, $f_{\hat{\xi}_n} \rightarrow \varphi_\ell$ 
uniformly on $\Real^\ell$ with probability one. 
\end{theorem}
\begin{proof}
Let $B_n$ be the event that $\hat{\xi}_n$ satisfies \eqref{eq:tvmain}. Under this event, by Theorem~\ref{thm:klartag}, for $n\geq \ell^{1/c_1},$
$$\sup_{x: \norm{x} \leq n^{c_4}} |f_{\hat{\xi}_n}(x) - \varphi_\ell(x)| \leq \frac{C}{n^{c_3}}.\sup_z\varphi_\ell(z),$$
with probability at least $1-C\exp(-n^{c_2}).$
Set $K'=K'(\ell) =C \sup_z\varphi_\ell(z)=\frac{C}{(2\pi)^{\ell/2}}$. 
Recall that by Lemma~\ref{lem:subexp}, $f_{\hat{\xi}_n}(x) \leq \exp (-\frac{a}{\sqrt{2}}\norm{x}+b)$ for all $x\in \Real^\ell$ and for all $n.$ Meanwhile, $\varphi_\ell(x)=\exp(-\norm{x}^2/2-\ell \ln (2\pi)/2).$ Therefore, for $\norm{x}>1,$ one can find $a',b'$ depending only on $a,b,\ell$ such that $f_{\hat{\xi}_n}(x)\leq e^{-a'\norm{x}+b'}$ and $\varphi_\ell(x) \leq e^{-a'\norm{x}+b'}$ and hence, 
$$|f_{\hat{\xi}_n}(x) - \varphi_\ell(x)|   \leq e^{-a'\norm{x}+b'},$$ for all $x$ such that $\norm{x}>1.$ Applying this bound for $\norm{x}>n^{c_4}$ and with the earlier bounds gives the first result.

Let $\ell'\geq \ell$. $f_{\hat{\xi}_n} \rightarrow \varphi_\ell$ uniformly under the event $B=\cap_{n\geq (\ell')^{1/c_1}} B_n$. Following the proof of Lemma~\ref{lem:pointwisedensity}, $B$ holds with probability at least $1-K\exp(-(\ell')^{c_2/c_1}).$ Letting $\ell' \rightarrow \infty,$ we get that $f_{\hat{\xi}_n} \rightarrow \varphi_\ell$ uniformly $\{R_n\}$-a.s.
\end{proof}

\section{Conclusion} \label{sec:conc} 
We considered stochastic static team problems with quadratic cost and linear information but with noise vectors that have densities that are not necessarily Gaussian. These problems have been studied extensively in the case where noise is Gaussian, where it is known that they admit linear optimal controllers. We considered  noise vectors to be with log-concave densities and established that for most problems of such a kind, linear strategies are near-optimal. Further, if the optimal strategies converge pointwise as the noise vector grows in length, they do so to a linear strategy. We derived an asymptotically tight bound on the difference between the optimal cost and the cost under the best linear strategy. 
Our results were propelled by a recent central limit theorem of Eldan and Klartag~\cite{eldan2008pointwise}. 
Additionally, we derived subsidiary results on uniform convergence of projections of random vectors with log-concave densities and on uniform convergence of tails of optimal costs with non-Gaussian noise.



\section*{Acknowledgment}
The author would like to thank Prof Vivek S. Borkar who introduced him to the work of Eldan and Klartag, Prof Jayakrishnan Nair and Prof Abhishek Gupta for some helpful discussions and user PhoemueX~\footnote{\url{http://math.stackexchange.com/users/151552/phoemuex}}  on StackExchange.com for the proof idea for Lemma~\ref{lem:phoemuex}~\cite{SEweakly}. He would also like to thank Prof Bo'az Klartag for clarifying some aspects of his result. This work supported in part by a grant of the Science and Engineering Research Board, Department of Science and Technology, Government of India. The author would also like to thank the Senior Editor and the anonymous reviewers for their comments.

\begin{appendix}
\subsection{Proof of Lemma~\ref{lem:pointwisedensity}}
\begin{proof}
Let $\ell' \geq \ell.$
For each $n \geq (\ell')^{1/c_1}$, let $B_n$ be the event that, $
B_n:= \left\{\hat{\xi}_n \ {\rm satisfies\ \eqref{eq:tvmain}} \right\}$.
We claim that the event $A:=\left\{ \limn f_{\hat{\xi}_n}(x) = \varphi_\ell(x)\  \forall \ x \in \Real^\ell\right\}  $ is implied by the event $B:=\cap_{n \geq (\ell')^{1/c_1}} B_n$. Consider an $x\in \Real^\ell$. Then under the event $B$, 
\[|f_{\hat{\xi}_n}(x) - \varphi_\ell(x)| \leq \frac{C}{n^{c_3}}, \quad \forall n \geq \norm{x}^{1/c_4},\]
and hence $f_{\hat{\xi}_n}(x) \rightarrow \varphi_\ell(x)$. However, since this holds for every $x\in \Real^\ell$, it follows that under event $B$, $f_{\hat{\xi}_n} \rightarrow \varphi_\ell$ pointwise.
Therefore $P(A) \geq P(B) \geq 1-\sum_{n \geq (\ell')^{1/c_1}} P(B_n^c).$ Now,
\begin{align*}
\sum_{n \geq (\ell')^{1/c_1}} P(B_n^c) &\leq C\sum_{n\geq (\ell')^{1/c_1}} \exp(-n^{c_2}) \\
& \leq C\sum_{n\geq (\ell')^{c_2/c_1}} \exp(-n), 
\end{align*}
and the RHS evalutates to $K\exp(-(\ell')^{c_2/c_1})$ with $K=\frac{C}{1-e\inv}$. Notice that this is true for any $\ell' \geq \ell$. It follows that \[P(A) \geq \sup_{\ell'} \{1 -K\exp(-(\ell')^{c_2/c_1})\} =1,\]
where the last equality follows since $c_2,c_1>0$. 
This completes the proof.
\end{proof}
\subsection{Proof of Lemma~\ref{lem:marginals}}
\begin{proof}
If $W$ is nonsingular (\ie, $\ell=r$), then, $\norm{Wx} \leq \half \sigma_{\min} n^{c_4}$ implies $\norm{x} \leq \half n^{c_4}$. Further, $f_{W\hat{\xi}}(Wx) \equiv \frac{1}{|\det(W)|}f_{\hat{\xi}}(x)$ and $f_{W\zeta}(Wx)\equiv \frac{1}{|\det(W)|}f_\zeta(x)=\frac{1}{|\det(W)|}\varphi_\ell(x),$ whereby the result holds by a direct application of \eqref{eq:b1}.

Now assume $W$ is singular and of rank $r$. Consider the singular value decomposition of $W$, 
$W=U\Sigma V\t$ where $U,V \in V_{\ell,\ell}$, 
$\Sigma = \pmat{\Sigma_1 & 0 \\ 0& 0}$ and $\Sigma_1$ is a $r\times r$ diagonal matrix with positive diagonal entries (the smallest of which is $\sigma_{\min}$). Let $V=[V_1; V_2]$ where $V_1 \in \Real^{\ell \times r}.$ It is easy to check that for any $x \in \Real^\ell$, $\norm{Wx}=\norm{\Sigma_1 V_1\t x} $ (the latter norm is on $\Real^r$). Consequently, $\norm{Wx}\leq \half \sigma_{\min} n^{c_4}$ implies that $\norm{V_1\t x} \leq \half n^{c_4}.$

Note that for any $x\in \Real^\ell,$
\begin{align}
f_{W\zeta}(Wx) &
=f_{\Sigma V\t \zeta}(\Sigma V\t x) 
=f_{\Sigma V\t \zeta}(\Sigma_1 V_1\t x,0)\non \\ &= \frac{ f_{V_1\t \zeta}( V_1\t x)}{|\det(\Sigma_1)|}, \label{eq:fw1} 
\end{align}
and similarly, 
\begin{equation}
f_{W\hat{\xi}}(Wx)
= \frac{1}{|\det(\Sigma_1)|} f_{V_1\t \hat{\xi}}( V_1\t x).\label{eq:fw2} 
\end{equation}
Since $V$ is orthogonal, for any $z\in \Real^r,$ we have
\begin{align}
f_{V_1\t \hat{\xi}}(z) &= \int_{\Real^{\ell-r}} f_{V\t \hat{\xi}}(z,\bar{z})d\bar{z} = \int_{\Real^{\ell-r}} f_{\hat{\xi}}(V(z,\bar{z}))d\bar{z}, \non \\
f_{V_1\t \zeta}(z) &= \int_{\Real^{\ell-r}} f_{V\t \zeta}(z,\bar{z})d\bar{z} = \int_{\Real^{\ell-r}} \varphi_\ell(V(z,\bar{z}))d\bar{z} \non \\
&=\int_{\Real^{\ell-r}}\varphi_\ell(z,\bar{z})d\bar{z}= \varphi_r(z), \label{eq:temp4} 
\end{align}
where the last relation follows from the spherical symmetry of the normal distribution. For $z\in \Real^r$, let $A_1=A_1(z):=\{\bar{z} \in \Real^{\ell -r}:\norm{z,\bar{z}} \leq n^{c_4}\}$ and 
$A_2=A_2(z):=\{\bar{z} \in \Real^{\ell -r}:\norm{z,\bar{z}} >n^{c_4}\}$. 
For any $z\in \Real^r$, we get, 
\begin{align}
f_{V_1\t \zeta}(z) -f_{V_1\t \hat{\xi}}(z)  &\buildrel{(a)}\over \leq \int_{\bar{z}\in A_1}\hspace{-3mm}\left( \varphi_\ell(V(z,\bar{z})) -f_{\hat{\xi}}(V(z,\bar{z})) \right) d\bar{z} \non \\
&\pushright{+ \int_{\bar{z}\in A_2} \varphi_\ell(V(z,\bar{z}))d\bar{z}} \non \\
&\buildrel{(b)}\over\leq \frac{C}{n^{c_3}}\int_{\bar{z}\in A_1} \varphi_\ell(V(z,\bar{z}))d\bar{z}\non \\
&\hspace{1cm}+ \int_{\bar{z} \in A_2}\varphi_\ell(V(z,\bar{z}))d\bar{z} \label{eq:temp3} 
\end{align}
where $(a)$ is obtained by using that $-\int_{\bar{z} \in A_2} f_{\hat{\xi}}(V(z,\bar{z}))d\bar{z} \leq 0 $ and $(b)$ follows from \eqref{eq:b1}. Now, if $\norm{z} \leq \half n^{c_4},$ we get $A_2(z) \subseteq A_2':=\{\bar{z} \in \Real^{\ell-r} | \norm{\bar{z}}^2>\frac{3}{4}n^{2c_4}\}$. Therefore, the second term in \eqref{eq:temp3} is at most,
$\varphi_r(z)\int_{\bar{z} \in A_2'} \varphi_{\ell-r}(\bar{z})d\bar{z}
 = \tau_{\ell,r}(n^{c_4})\varphi_r(z).$
The first term in \eqref{eq:temp3} is at most $\frac{C}{n^{c_3}}\int_{\Real^{\ell-r}}\varphi_\ell(z,\bar{z})d\bar{z} =\frac{C}{n^{c_3}}\varphi_r(z).$ 
Thus, combining this with \eqref{eq:temp3}, \eqref{eq:temp4}, \eqref{eq:fw1}, \eqref{eq:fw2}, the result follows.
%
\end{proof}

\subsection{Proof of Theorem~\ref{thm:error}}
\begin{proof}
Let $n\geq \ell^{1/c_1}$ and let $\Delta_n:= J(\gamma^*_L;\xi_n)-J(\gamma^*_n;\xi_n).$ We have $\Delta_n \geq 0,$ since $\gamma^*_n$ is the optimal controller. Since $J(\gamma^*_G;\zeta)=J(\gamma^*_L;\xi_n)$ (cf. Propositions~\ref{prop:gammaLG} and \ref{prop:indR}) we get
\begin{align*}
\Delta_n&=J(\gamma_G^*;\zeta)-J(\gamma^*_n;\xi_n) 
\leq J(\gamma';\zeta)-J(\gamma^*_n;\xi_n), 
\end{align*}
for any measurable function $\gamma'.$  Let $A,\bar{\gamma}$ be as defined in the theorem and take $\gamma'=\bar{\gamma}.$ 
Therefore, $\gamma'(w)=\gamma^*_n(w)$ for all $w \in A$. Hence,
\begin{align}
&\hspace{-1mm}\Delta_n \leq 
\int_{A}\hspace{-1mm} \left( L(\gamma'(w),w)
f_{W\zeta}(w) - L(\gamma^*_n(w),w)f_{W\hat{\xi}_n}(w) \right) dw \non \\
&+ \int_{A^c}\left( L(\gamma'(w),w)
f_{W\zeta}(w) - L(\gamma^*_n(w),w)f_{W\hat{\xi}_n}(w) \right)dw, \non \\  
&\leq \int_{A} L(\gamma_n^*(w),w) 
\left(f_{W\zeta}(w) - f_{W\hat{\xi}_n}(w) \right) dw \non \\
&\qquad+ \int_{A^c}\hspace{-1mm}L(\bar{\gamma}(w),w)f_{W\zeta}(w)dw
\label{eq:temp} 
\end{align}
where 
`$dw$' denotes the Lebesgue measure on the range of $W$ and we have used that 
$-\int_{A^c}L(\gamma^*_n(w),w)f_{W\hat{\xi}_n}(w)dw \leq 0,$
(recall that $L(\cdot,\cdot) \geq 0$). 
For the first term in \eqref{eq:temp}, consider the event $B_n$ that $\hat{\xi}_n$ satisfies \eqref{eq:tvmain}. Applying Theorem~\ref{thm:klartag}, we get that with probability at least $1-C\exp(-n^{c_2})$,  
\eqref{eq:b1} holds with $\hat{\xi}=\hat{\xi}_n$. Then since $w \in A \implies \norm{w} \leq \half \sigma_{\min}n^{c_4},$  by Lemma~\ref{lem:marginals}, we have that with probability at least $1-C\exp(-n^{c_2}),$
\begin{align}
0\leq \Delta_n &\leq 
\left(\frac{C}{n^{c_3}} +\tau_{\ell,r}(n^{c_4})\right) \int_{A} L(\gamma_n^*(w),w)
f_{W\zeta}(w) dw \non \\
&+\int_{A^c}L(\bar{\gamma}(w),w)f_{W\zeta}(w)dw
\label{eq:temp2} 
\end{align}
Notice that the first term in \eqref{eq:temp2}  equals
\begin{align*}
& \left(\frac{C}{n^{c_3}}+\tau_{\ell,r}(n^{c_4})\right) \int_{A} L(\gamma_n^*(w),w)
\frac{f_{W\zeta}(w)}{f_{W\hat{\xi}_n}(w)}f_{W\hat{\xi}_n}(w) dw, 
\end{align*}
which, if the event $B_n$ holds and 
 if $1- \frac{C}{n^3}-\tau_{\ell,r}(n^{c_4})>0$, is, by \eqref{eq:marginal}, 
at most 
\begin{align*} 
\frac{\frac{C}{n^{c_3}}+\tau_{\ell,r}(n^{c_4})}{\left(1- \frac{C}{n^3}-\tau_{\ell,r}(n^{c_4})\right)} &\int_{A} L(\gamma_n^*(w),w)
f_{W\hat{\xi}_n}(w) dw \\
&\leq \frac{\frac{C}{n^{c_3}}+\tau_{\ell,r}(n^{c_4})}{\left(1- \frac{C}{n^3}-\tau_{\ell,r}(n^{c_4})\right)}J(\gamma^*_G;\zeta).
\end{align*}
The last inequality follows from the nonnegativity of $L$ whereby $\int_{A} L(\gamma_n^*(w),w)
f_{W\hat{\xi}_n}(w) dw \leq J(\gamma^*_n;\xi_n)\leq J(\gamma^*_G;\xi_n)=J(\gamma^*_G;\zeta).$ This completes the proof.
\end{proof}

\end{appendix}



%

\bibliographystyle{IEEEtran}
\bibliography{../Mylatexfiles/ref.bib}

\begin{IEEEbiography}[{\includegraphics[width=1in]{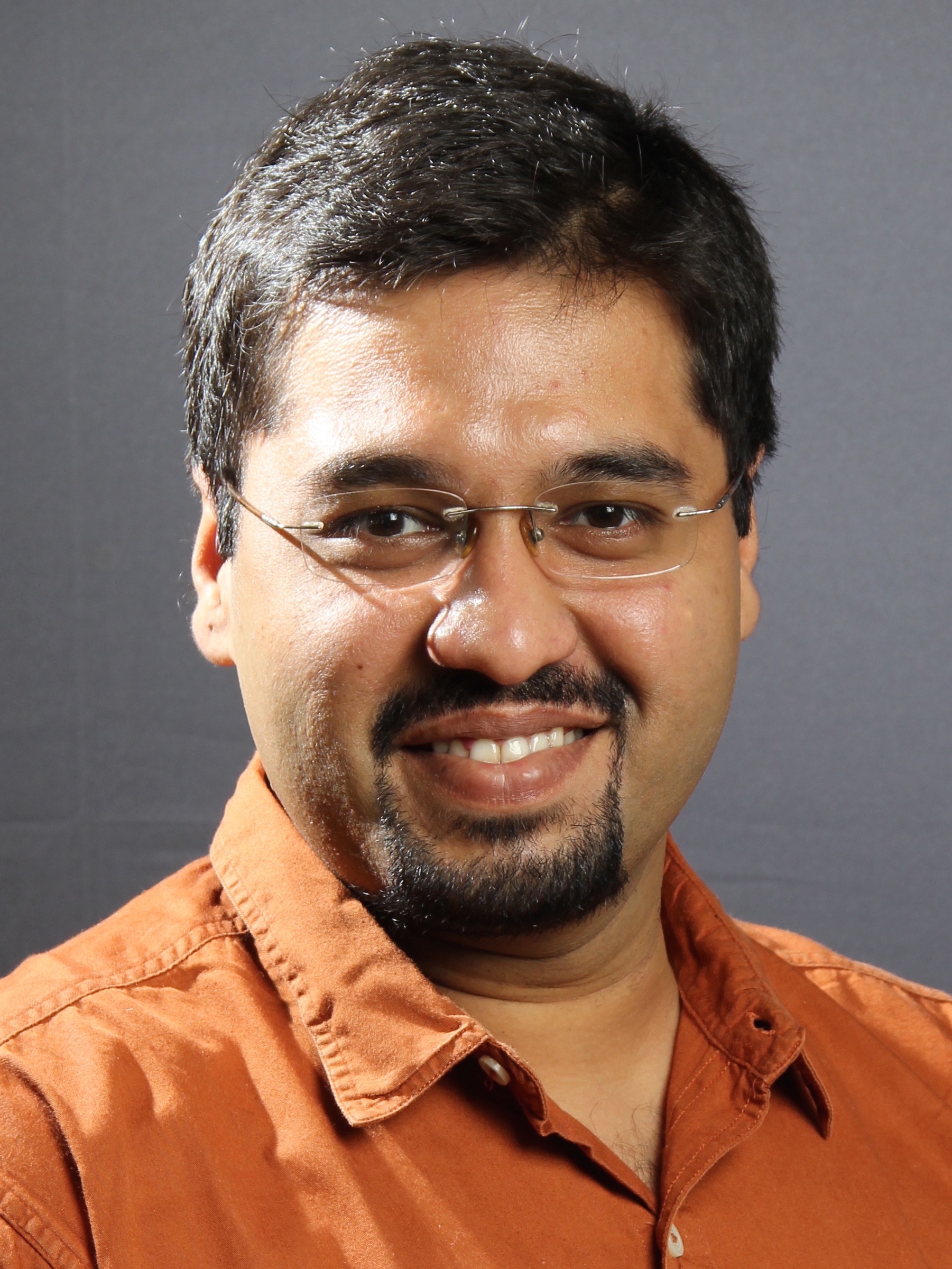}}]{Ankur A. Kulkarni}
Ankur is an Associate Professor with the Systems and Control Engineering group at Indian Institute of Technology Bombay (IITB). He received his B.Tech. from IITB in 2006, M.S. in 2008 and Ph.D. in 2010, both from the University of Illinois at Urbana-Champaign (UIUC). From 2010-2012 he was a post-doctoral researcher at the Coordinated Science Laboratory at UIUC. His research interests include information theory, the role of information in stochastic control, game theory, combinatorial coding theory problems, optimization and variational inequalities, and operations research. He is an Associate (from 2015--2018) of the Indian Academy of Sciences, Bangalore, a recipient of the INSPIRE Faculty Award of the Department of Science and Technology, Government of India, 2013, Best paper awards at the National Conference on Communications, 2017, Indian Control Conference, 2018, International Conference on Signal Processing and Communications (SPCOM) 2018, Excellence in Teaching Award 2018 at IITB  and the William A. Chittenden Award, 2008 at UIUC. He is a consultant to the Securities and Exchange Board of India on regulation of high frequency trading.
\end{IEEEbiography}

\end{document}